\documentclass[11pt]{amsart}

\usepackage{amssymb,amsfonts,amsthm,amsmath,enumerate,amscd}
\usepackage[english]{babel}
\usepackage{colordvi,color}
\usepackage{graphicx}
\usepackage{mathabx}

\newtheorem{theorem}{Theorem}[section]

\newtheorem{lemma}[theorem]{Lemma}
\newtheorem{corollary}[theorem]{Corollary}
\newtheorem{claim}[theorem]{Claim}
\newtheorem{proposition}[theorem]{Proposition}
\newtheorem{definition}[theorem]{Definition}
\newtheorem{example}[theorem]{Example}
\newtheorem{remark}[theorem]{Remark}

\newcommand{\field}[1]{\mathbb{#1}}

\newcommand{\R}{\field{R}}

\newcommand{\wh}[1]{\widehat{#1}}
\newcommand{\wt}[1]{\widetilde{#1}}
\newcommand{\aph}{{\alpha}}
\newcommand{\ba}{{\bf a}}
\newcommand{\bb}{{\bf b}}
\newcommand{\bbo}{{\bf 0}}
\newcommand{\bB}{{\bf B}}

\newcommand{\bc}{{\bf c}}

\newcommand{\bg}{{\bf g}}

\newcommand{\bt}{{\bf t}}

\newcommand{\bP}{{\bf P}}

\newcommand{\bs}{{\bf s}}
\newcommand{\bS}{{\bf S}}

\newcommand{\bu}{{\bf u}}

\newcommand{\bv}{{\bf v}}

\newcommand{\bx}{{\bf x}}

\newcommand{\by}{{\bf y}}
\newcommand{\bz}{{\bf z}}
\newcommand{\cB}{\mathcal{B}}
\newcommand{\cC}{\mathcal{C}}
\newcommand{\cH}{\mathcal{H}}
\newcommand{\cM}{\mathcal{M}}

\newcommand{\cS}{\mathcal{S}}
\newcommand{\cU}{\mathcal{U}}
\newcommand{\clos}{{\rm clos}}

\newcommand{\dist}{{\rm dist}}
\newcommand{\dlt}{{\delta}}
\newcommand{\Dlt}{{\Delta}}

\newcommand{\gm}{{\gamma}}
\newcommand{\Gm}{{\Gamma}}

\newcommand{\gR}{{\geq R}}

\newcommand{\inn}{{\rm in}}
\newcommand{\iop}{{\iota_p}}
\newcommand{\ioq}{{\iota_q}}

\newcommand{\lbd}{{\lambda}}
\newcommand{\lgt}{{\ell}}
\newcommand{\lr}{{\leq r}}
\newcommand{\lR}{{\leq R}}
\newcommand{\omg}{\omega}

\newcommand{\rd}{{\rm d}}

\newcommand{\Rgo}{{\R_{\geq 0}}}
\newcommand{\Rn}{{\R^n}}
\newcommand{\Rp}{{\R^p}}

\newcommand{\Rq}{{\R^q}}

\newcommand{\sgm}{\sigma}
\newcommand{\tht}{{\theta}}

\newcommand{\ve}{\varepsilon}

\newcommand{\wtap}{\widetilde{\alpha}}
\newcommand{\wtbe}{\widetilde{\beta}}

\newcommand{\wtgm}{{\widetilde{\gm}}}
\newcommand{\wtphi}{{\widetilde{\phi}}}
\newcommand{\wtX}{\widetilde{X}}

\numberwithin{equation}{section}

\numberwithin{equation}{section}

\begin{document}
\title[One point compactification of LNE definable sets]{One point compactification and Lipschitz normally embedded definable subsets}

\author[A. Costa]{Andr\'e Costa}
%
\author[V. Grandjean]{Vincent Grandjean}

\author[M. Michalska]{Maria Michalska}
\address{A. Costa, Departamento de Matem\'atica, 
Universidade Federal do Cear\'a
(UFC), Campus do Pici, Bloco 914, Cep. 60455-760. Fortaleza-Ce,
Brasil}
\email{andrecosta.math@gmail.com}
\address{V. Grandjean, Departamento de Matem\'atica, 
Universidade Federal de Santa Catarina, 
88.040-900 Florianópolis - SC, Brasil}
\email{vincent.grandjean@ufsc.br}
\address{M. Michalska, Wydzia\l{} Matematyki i Informatyki, Uniwersytet 
\L{}\'o{}dzki, Banacha 22, 90-238 \L{}\'o{}d\'z{}, Poland}
\email{maria.michalska@wmii.uni.lodz.pl}

\subjclass[2000]{}




\begin{abstract}
A closed subset of $\Rq$, definable in some given o-minimal structure, is Lipschitz normally embedded in $\Rq$ if and only if its one-point compactification is Lipschitz normally embedded in the unit sphere $\bS^q$($ = \Rq \cup \{\infty \}$), i.e. it is the 
stereographic projection of a set Lipschitz normally embedded in~$\bS^q$. This implies that 
 any closed connected unbounded definable subset of an Euclidean space is
 definably inner bi-Lipschitz homeomorphic to a Lipschitz normally embedded definable set.
\end{abstract}

\maketitle
\tableofcontents


\section*{Introduction}
One can equip any subset $S$ of $\Rq$ with two metric space structures induced by
the ambient metric: 
the outer metric structure $(S,d_S)$, where the (outer) distance $d_S$ is the 
restriction of the euclidean distance in $\Rq$, and the inner metric structure 
$(S,d_\inn^S)$, where the (inner) distance is the infimum of lengths of 
rectifiable curves contained in $S$ joining two given points. 
A classical problem introduced by Whitney \cite{Whi1,Whi2} is to ask when the 
two metric space structures are equivalent, i.e. when does there exist a 
positive constant $L$ such that
$$
d_S \; \leq \; d_\inn^S \; \leq \; L\cdot d_S.
$$
Since the influential paper \cite{BiMo} subsets which satisfy this property
are called Lipschitz Normally Embedded (shortened to LNE).
In the last decade there were quite a few works published on 
germs admitting a LNE representative, almost exclusively for complex 
analytic curve and surface singularities (see 
\cite{MeSa,FaPi} for an overview). Although there is a growing interest in Lipschitz geometry of affine
algebraic sets, authors always deal with the behaviour at 
infinity separately.
Characterization of LNE 
sets at infinity was addressed recently in the very restrictive case in 
\cite{FeSa2} and  in \cite{DiRi} presenting a necessary condition in terms of 
tangent 
cones. Yet prior to the first author PhD's results \cite{Cos} and our papers \cite{CoGrMi1, CoGrMi2} which show that LNE is a generic property of algebraic sets, 
the only non-trivial examples of 
(unbounded) algebraic LNE sets were provided in \cite{KePeRu}. 

The main motivation of this paper is to provide a convenient 
framework to study LNE subsets at infinity.  
In this article we prove that in the category of definable sets of $\Rq$ one can equivalently study their LNE property under one-point compatification. More 
precisely, consider
$$
\sgm_q :\Rq \to \bS^q\setminus N_q, \;\; {\rm for} \;\;
N_q=(0,\ldots,0,1) \in\R^{q+1},
$$
the inverse of the stereographic projection centred at $N_q$.
The main result of the paper is: 

\medskip\noindent
{\bf Theorem \ref{thm:main}.} \em 
A closed definable subset $X$ is LNE in $\Rq$ if and only if
$\clos(\sgm_q(X))$ is LNE in $\bS^q$. \em

\medskip
Using Theorem~\ref{thm:main} one can apply all standard local methods to the global problem, as a sample we give Corollaries~\ref{cor:LNE-inversion} and \ref{cor:LNE-inversionLINK}. The main result of this paper is indeed very natural as illustrated by the fact that some time after publishing this article, another preprint~\cite{Sam} reproved our Theorem \ref{thm:main} and Corollary  \ref{cor:LNE-inversionLINK}, as well as a result of second author of~\cite{GrOl}, as far as we know its methods differ from ours.
To further demonstrate the power of this tool, as its rather straightforward application we can remove the assumption of compactness from the main
result of \cite{BiMo} and obtain the result as follows

\medskip\noindent
{\bf Theorem \ref{prop:LNE-model}.} \em 
For every closed connected  definable subset $X$ of $\Rp$ there exists a connected closed 
definable  subset $Y$ of $\Rq$ which is LNE and  definably inner bi-Lipschitz homeomorphic to $X$.	\em

\medskip
The paper is organized as follows. After introducing  some
necessary notions in Section \ref{section:P}, we investigate pairs of definable curves and their asymptotic behaviour 
at $\bbo$ or $\infty$  in Section \ref{section:IODDA}. 
Section~\ref{section:L-RD} presents fundamental results for the rest of the article: Proposition \ref{prop:defin-metrics-param} and Corollary 
\ref{cor:repres-LNE-local} about LNE representatives of germs.  Section  
\ref{section:LNEaI} studies properties of the unbounded part of a
definable LNE subset. Section \ref{section:LLNEaIvsO} presents our second 
essential tool: the inversion taking the origin to infinity. The main result in Theorem \ref{thm:main} is 
proved in Section~\ref{section:MR} and its application to inner Lipschitz classification is Section \ref{section:LNEmodels}. In the last Section
\ref{section:examples} we provide some examples of interest.


%
%
%
%
%
%
%
%
%
%
%
%
%
%
%
%
%
%
%
%
%
%
%
%
%
%
%
%
%

\section{Preliminaries}\label{section:P}

	Throughout the paper we fix an o-minimal structure $\cM$ expanding the 
	ordered real field $(\R,+,\cdot,\geq)$.  
	In the sequel the adjective \em definable \em means definable in $\cM$
	(see \cite{vdDMi,vdD}).

The Euclidean space $\Rq$ is equipped with the Euclidean distance, denoted
$|-|$. We denote by $B_r^q$ the open ball of $\Rq$ of radius $r$ and centred 
at the origin $\bbo$, by $\bB_r^q$ its closure and by $\bS_r^{q-1}$ its 
boundary. The open ball of radius $r$ and centre $\bx_0$ is $B^q(\bx_0,r)$, 
its closure is $\bB^q(\bx_0,r)$ and $\bS^{q-1}(\bx_0,r)$ is its boundary.
The unit sphere of $\Rq$ is $\bS^{q-1}$.

Let $\Rgo := [0,\infty)$. The half-line in the oriented direction (of the 
vector) $\bu\in\Rq\setminus \bbo$ is 
$$
\Rgo \bu := \{t\bu : t \in \Rgo\}.
$$
The non-negative cone over the subset $S$ of $\Rq$ with vertex $\bx_0$ is 
defined as 
$$
\wh{S}^+ := \bx_0 + \cup_{\bx\in S\setminus\bx_0}\,\Rgo(\bx-\bx_0).
$$

\bigskip
Let $(M,d_M)$ be a metric space such that any two points can be joined by a rectifiable curve.

\begin{definition}\label{def:curves}
	An arc in  $M$ is a continuous 
	mapping $\gm:I\to M$ over a real interval $I$ of $\R$.
	
	The length $\lgt(\gm)$ of an arc $\gm: I\to M$ is
$$\lgt(\gm):= \sup \sum_{j=1}^m d_M(\gm(t_{j-1}), \gm(t_{j})), $$
where the supremum is taken over all sequences $t_0<\dots< t_m$ in the interval $I$.
\end{definition}

Any subset $S$ of $M$ admits two natural metric space structures 
inherited from $(M,d_M)$:
\begin{definition}\label{def:outer-inner}
1) The \em outer metric space structure $(S,d_S)$, \em is equipped with 
the \em outer distance \em function $d_S$, 
restriction of $d_M$ to $S\times S$. 

\smallskip\noindent
2) The \em inner metric space structure $(S,d_\inn^S)$, \em is equipped with
the \em inner distance \em function 
$$
d_\inn^S: S\times S \to [0,+\infty]
$$
defined as follows: given $\bx,\bx'\in S$, the number $d_\inn^S(\bx,\bx')$ is the infimum of the lengths
of the rectifiable paths lying in $S$ joining $\bx$ and $\bx'$.   
\end{definition}
Observe that $d_S \leq d_\inn^S$ and the length of an arc remains the same regardless whether it is taken with respect to the inner or outer metric. 

\begin{definition}\label{def:LNE}
i) A subset $S$ of $(M,d_M)$ is \em Lipschitz normally 
embedded \em (shortened to LNE) if 
there exists a positive constant $L$ such 
that
$$
\bx,\bx' \; \in \; S \; \Longrightarrow \; d_\inn^S(\bx,\bx') \; \leq \;
L \cdot d_S(\bx,\bx').
$$
Any constant $L$ satisfying the previous inequality is called a \em LNE 
constant of $S$. \em 

\smallskip
(ii) The subset $S$ of $(M,d_M)$ is \em locally LNE at $\bx_0$ \em if 
there exists a neighbourhood $U$ of $\bx_0$ in $M$ such that
$S\cap U$ is LNE.
\end{definition}

\noindent
Note that LNE property of the subset depends on the ambient metric of the space and in general the ambient metric space itself may not be LNE. 
An important  class of  metric spaces $(M,d_M)$ are those where $d_M = 
d_\inn^M$, that is the metric $d_M$ is a length metric. 
Connected Riemannian manifolds are examples of such spaces. Indeed, given 
a connected Riemannian manifold $(M,g_M)$, it is naturally equipped with the length 
metric $d_M = d_\inn^M$ induced by its metric tensor $g_M$. The next result characterizes completely compact LNE subsets in Riemannian manifolds
(see \cite[Proposition 2.4]{KePeRu} and \cite[Lemma 2.6]{CoGrMi1}).
\begin{lemma}\label{lem:compact-LNE}
A connected compact subset of the $C^\infty$ Riemannian manifold $(M,g_M)$ is 
LNE if and only if it is locally LNE
at each of its points.
\end{lemma}

In the Euclidean space $\Rq$ every complement of a compact set is considered to be the neighborhood of infinity. Thus we extend Definition~\ref{def:LNE} accordingly:
\begin{definition}\label{def:LNE-infty}
The subset $S$ of $\Rq$ is \em locally LNE at $\infty$ \em if 
there exists a neighbourhood~$U$ of $\infty$ in $\Rq$ such that
$S\cap U$ is LNE.
\end{definition}%

Throughout this paper we will work principally with closed subsets, because 
obviously if a set~$S$ is LNE, then $\clos(S)$ is LNE with the same constant.

%
%
%
%
%
%
%
%
%
%
%
%
%
%
%
%
%
%
%
%
%
%
%
%
%
%
%
%
%
%
%
%
%
\section{Inner and outer distances and definable arcs}\label{section:IODDA}

We recall some results of \cite{dAc} about definable curves. Let $Y$ be a 
closed definable subset of $\Rq$. Any definable arc $\gm : [0,1] \to Y$ 
is rectifiable since it is piecewise $C^k$, for any $k\geq 1$.

We can assume without loss of generality that $\gm(0) = \bbo$. 
Let $g(t) := |\gm(t)|$. There exist a unit 
vector $\bu\in\bS^{q-1}$ and a continuous definable path $\bv:[0,1] \to \Rq$ 
such that
\begin{equation}\label{eq:arc}
\gm(t) = g(t)(\bu + \bv(t)) \;\; {\rm with} \;\; \bv(t) \to \bbo \;\; 
{\rm as} \;\; t\to 0.
\end{equation}
Moreover $\gm$ is $C^k$ over $(0,a_k)$ for some $a_k>0$
and \cite[Lemma 2.6]{dAc} yields
\begin{equation}\label{eq:arc-derivative}
\gm'(t) = g'(t) \bu + o(g'(t)) \;\; {\rm as} \;\; t\to 0.
\end{equation}
Similarly, let $\gm : [1,\infty) \to Y$ be an unbounded definable arc such 
that $g(t) :=|\gm(t)| \to \infty$ as $t$ goes to $\infty$. It is also 
piecewise $C^k$ for any $k\geq 1$. There exist a unit vector 
$\bu\in\bS^{q-1}$ 
and a continuous path $\bv:[1,\infty) \to \Rq$ such that
\begin{equation}\label{eq:arc-infty}
\gm(t) = g(t)(\bu + \bv(t)) \;\; {\rm with} \;\; \bv(t) \to \bbo  \;\; 
{\rm as} \;\; t\to \infty.
\end{equation}
The arc $\gm$ is $C^k$ over $(A_k,\infty)$ for some $A_k>0$ and 
we still have
\begin{equation}\label{eq:arc-infty-derivative}
\gm'(t) = g'(t) \bu + o(g'(t)) \;\; {\rm as} \;\; t\to \infty.
\end{equation}

\medskip
The next result is known, but we provide proof to show the methods.
\begin{lemma}\label{lem:inner-outer-arc}
Let $Y$ be a definable subset of $\Rq$. 

(i) Let $\bx_0$ be a point of $Y$ such that $\dim (Y,\bx_0) \geq 1$. 
Let $\gm :(0,1]\to Y\setminus\bx_0$ be any definable arc such that
$\gm(t)$ goes to $\bx_0$ as $t$ goes to $0$. The following estimate holds 
true 
$$
\lim_{t\to 0}\, \frac{\lgt(\gm(]0,t]))}{|\gm(t) - \bx_0|} = 1.
$$

(ii) Let $\gm : [1,\infty) \to Y$ be any definable arc such that
$|\gm(t)|$ goes to $\infty$ as $t$ goes to $\infty$. 
The following estimate holds 
true 
$$
\lim_{t\to \infty}\, \frac{\lgt(\gm([1,t]))}{|\gm(t)|} = 1.
$$ 
\end{lemma}
\begin{proof}
We start with (i). We can assume that $\gm$ extends definably and $C^1$ over 
$[0,r]$ for some positive $r\leq 1$. From Equation \eqref{eq:arc}, defining 
$g(t) := |\gm(t) - \bx_0|$, we can write
$$
\gm(t) - \bx_0 = g(t) ( \bu + \bv(t) ) = g(t)\bu + o(g(t)) \;\; {\rm as} 
\;\; t \to 0,
$$ 
as well as, in using Equation \eqref{eq:arc-derivative}, we get
$$
\gm'(t) = g'(t) \bu + o(g'(t)) \;\; {\rm as} 
\;\; t \to 0.
$$ 
We deduce that $\lgt(\gm([0,t])) = g(t) + o(g(t))$, proving the desired 
estimate.

\medskip\noindent
Similar arguments and estimates as $t$ goes to $\infty$, 
using Equations \eqref{eq:arc-infty} and \eqref{eq:arc-infty-derivative} 
yield (ii). 
\end{proof}
Let $\gm$ be an arc in $Y$ as in (i) or (ii) of Lemma 
\ref{lem:inner-outer-arc}. 
Let $(\aph,\omg) := (0,0)$ in the case (i) and $(\aph,\omg) := (1,\infty)$ 
in the case (ii).
Let $I_\omg := [0,1]$ if $\omg = 0$ and $I_\omg := [\aph,\infty)$ 
when $\omg =\infty$. Lemma \ref{lem:inner-outer-arc} implies the following 
estimates
\begin{equation}\label{eq:inner-infty}
\lim_{t\in I_\omg, t \to \omg} \frac{d_\inn^Y(\gm(\aph),\gm(t))}{|\gm(\aph) - 
\gm(t)|} = 1.
\end{equation}

\smallskip
%
%
\begin{lemma}\label{lem:asympt-u1-u2}
Let $Y$ be a definable connected subset of $\Rq$. Consider the two 
following settings:

(i) Let $\bx_0$ be a point of $Y$. Let $\by_1,\by_2 :]0,1] \to Y\setminus 
\bx_0$ be two definable paths such that $|\by_i - \bx_0| \to 0$
as $t\to 0$, for $i=1,2$. Define 
$$
\bu_1 := \lim_0 \frac{\by_1-\bx_0}{|\by_1-\bx_0|}, \;\; 
\bu_2 := \lim_0 \frac{\by_2-\bx_0}{|\by_2-\bx_0|} \;\; {\rm and} \;\; 
\lbd := \lim_0 \frac{|\by_2-\bx_0|}{|\by_1-\bx_0|} \in [0,\infty].
$$
If either $\bu_1 \neq \bu_2$ or $\bu_1 = \bu_2$ and $\lbd \neq 1$ is 
satisfied then
$$
\limsup_{t\to 0} \frac{d_\inn^Y(\by_1(t),\by_2(t))}{|\by_1(t) - \by_2(t)|} 
< \infty.
$$

(ii) Let $\by_1,\by_2 :[0,\infty) \to Y$ be two definable  paths 
such  that $|\by_i(t)| \to \infty$ as $t\to \infty$, for $i=1,2$. 
 Define 
$$
\bu_1 := \lim_\infty \frac{\by_1}{|\by_1|}, \;\; 
\bu_2 := \lim_\infty \frac{\by_2}{|\by_2|} \;\; {\rm and} \;\; 
\lbd := \lim_\infty \frac{|\by_2|}{|\by_1|} \in [0,\infty].
$$
If either $\bu_1 \neq \bu_2$ or $\bu_1 = \bu_2$ and $\lbd \neq 1$ is 
satisfied then
$$
\limsup_{t\to\infty} \frac{d_\inn^Y(\by_1(t),\by_2(t))}{|\by_1(t) - 
\by_2(t)|} < \infty.
$$
\end{lemma}
\begin{proof}
We deal first with the case (ii).

\medskip\noindent
$\bullet$ \em Suppose that $\bu_1 \neq \bu_2$. \em

\smallskip  
We can assume that $t= |\by_2(t)| \geq |\by_1(t)| =: y(t)$
for all $t$ large enough. 
We write 
$$
\by_1(t) = y(t) ( \bu_1 + o(1) ) \;\; {\rm and} \;\; \by_2(t) = t 
(\bu_2 + o(1) ), \;\; {\rm as} \;\; t \to \infty.
$$
Let $2\tht(t) \in [0,\pi]$ be the non-oriented angle between
$\by_1(t)$ and $\by_2(t)$. The law of cosines
\begin{equation}\label{eq:law-cosine}
e^2(t) := |\by_1(t) - \by_2(t)|^2 = (y(t)+t)^2 \sin^2 (\tht(t)) + 
(y(t) -t)^2 \cos^2 (\tht(t))
\end{equation}
yields $e^2(t)  \; \geq (y(t)+t)^2 \sin^2\tht(t)$.
For $t$ large enough the following
estimate holds true
$$
e(t) \; \geq \; \frac{y(t)+t}{2}\sin\tau,
$$
since $2\tht(t) \to 2\tau \in (0,\pi]$. Combining the next estimates for $t$ 
large enough  
$$
\lgt(\by_1([0,t])) = y(t) + o(y(t)) \;\; {\rm and} \;\; 
\lgt(\by_2([0,t])) = t + o(t).
$$
with the estimate \eqref{eq:inner-infty}, we deduce, after 
connecting $\by_1(0)$ to $\by_2(0)$ with a definable path, the following
$$
d_\inn^Y(\by_1(t),\by_2(t)) \; \leq \; 
\frac{4}{\sin\tau} \cdot e(t).
$$

\noindent
$\bullet$ \em Assume $\bu_1 =\bu_2$ and $\lbd > 1$. \em

\smallskip
Let $0 <\mu< 1$ such that, $y(t) \leq \mu t$ for large $t$. We observe that 
for large $t$ the following 
estimates are satisfied
$$
(1-\mu) t \; \leq \; e(t) := |\by_1(t) - \by_2(t)| \; \leq \;
(1+\mu) t. 
$$
Since $y,id_\R$ are $C^1$ and strictly increasing as $t\to \infty$, let 
$$
\bz_2 (t) := \by_2 (y(t))), 
$$
so that $|\bz_2| = y$. Since $\bz_2 = y(\bu_1 + o(1))$, we deduce that
$$
e_1:= |\by_1 - \bz_2| = o(y).
$$
Since $Y$ is connected, $Y_\lR := Y \cap \bB_R^q$ is connected for $R\geq R_0$, with $R_0$ large 
enough.
By \cite[Corollary 1.3]{KuPa} applied to the family $(Y_\lR)_R$, 
there exists a positive constant $C>0$ such 
that for each $R\geq R_0$ there exists a definable arc 
$\tht_R : [0,1] \to Y_{\leq y(R)}$ connecting $\by_1(R)$ and 
$\bz_2(R)$, such that
$$
\lgt(\tht_R) \; \leq C y(R).
$$
Therefore whenever $y (t) \geq R_0$, we deduce
$$
\lgt(\by_2 [y(t),t]) \; \leq  \; 2(t - y (t)).
$$
For $R$ such that $y(R) \geq R_0$, let $\gm_R$ be a rectifiable path
going from $\by_1(R)$ to $\bz_2(R)$ along $\tht_R$, and reaching
$\by_2(R)$ following $\by_2$. We deduce
$$
\lgt (\gm_R) \; \leq \; C y(R) + 2(R - y(R)) \; \leq \; 
\frac{C+2}{1-\mu} e(R).
$$

\medskip
The case (i) proceeds from similar arguments but as $t$ goes
to $0$. The case $\bu_1 \neq \bu_2$ is dealt with going through $\bx_0$ 
along $\by_1$, $\by_2$ an using the law of cosines. The first case
also connects $\by_1(0)$ to $\by_2(0)$ following $\by_1$ to $\bx_0$ and
then following $\by_2$. Estimates very similar to the corresponding case at 
$\infty$ give the result.
\end{proof}
We will need the following notions.
\begin{definition}\label{def:asmptotic-set}
Let $S$ be a subset of $\Rq$.   

(i) Let $\bx_0$ be a point in the closure $\clos(S)$ of $S$ taken in $\Rq$. 
The \em asymptotic set $S_{\bx_0} S$ of the subset $S$ at $\bx_0$ \em 
is defined as
$$
S_{\bx_0} S := \left\{\bu\in\bS^{q-1} \, : \, \exists (\bx_n)_n \subset 
S\setminus \bx_0 \;\; 
{\rm with} \;\; \bx_n \to \bx_0 \;\; {\rm and} \;\; 
\frac{\bx_n-\bx_0}{|\bx_n-\bx_0|} \to \bu\right\}.
$$

(ii) The \em asymptotic set $S^\infty$ of the subset $S$ at $\infty$ \em 
is defined as
$$
S^\infty := \left\{\bu\in\bS^{q-1} \, : \, \exists (\bx_n)_n \subset S \;\; 
{\rm with} \;\; |\bx_n| \to \infty \;\; {\rm and} \;\; 
\frac{\bx_n}{|\bx_n|} \to \bu\right\}.
$$  
\end{definition}

\medskip
Given a subset $S$ of $\Rq$ and $t>0$, we define the
following:
$$
S_{\leq t} := S \cap \bB_t^q, \;\; S_t := S \cap \bS_t^{q-1} \;\; {\rm and}
\;\; S_{\geq t} := S \setminus B_t^q .
$$

\medskip
Let $Y$ be a closed definable subset of $\Rq$. Let $\bx_0$ be a point of 
$Y$. Then $S_{\bx_0}Y$ is a definable set of dimension $\dim(Y,\bx_0)-1$ or
lower. In particular the non-negative cone $\wh{S_{\bx_0}Y}^+$ 
is the tangent cone of $Y$ at $\bx_0$. Similarly, $Y^\infty$ is definable and 
of dimension $\dim(Y,\infty)-1$ or lower. Often the non-negative cone
$\wh{Y^\infty}^+$ is called the tangent cone of $Y$ at infinity.
Note that
$$
S_{\bx_0} Y := \lim_{r\to\infty} \frac{1}{r} \left( Y \cap \bS^{q-1}
(\bx_0,r) \right) \;\; {\rm and} \;\; Y^\infty := \lim_{R\to\infty} 
\frac{1}{R} Y_R.
$$
where each limit is taken in the Hausdorff sense in $\bS^{q-1}$. 
We also get
$$
S_{\bx_0} Y = S_{\bx_0} (Y \cap \bB^q(\bx_0,r)) \;\; {\rm and} \;\; 
Y^\infty = (Y_\gR)^\infty.
$$
The conical structure of $Y$ at $\bx_0$ following \cite{vdDMi,vdD}
states that there exist a small radius $r_0 = r_0(\bx_0)$ and a 
homeomorphism $h :Y \cap \bB^q(\bx_0,r_0) \setminus \bx_0 \to (Y \cap 
\bS^{q-1}(\bx_0,r_0)) \times ]0,r_0]$
such that
$$
0 < r \leq r_0 \; \Longrightarrow \; h(Y\cap 
\bS^{q-1}(\bx_0,r)) = Y \cap 
\bS^{q-1}(\bx_0,r_0) \times r.
$$ 
Let $Y_1,\ldots,Y_c,$ be the closures of the connected components 
of $Y\cap \bB^q(\bx_0,r_0)\setminus \bx_0$. Each $Y_i$ is a closed connected 
definable subset containing $\bx_0$. We further find 
$$
S_{\bx_0} Y := \cup_{i=1}^c S_{\bx_0} Y_i.
$$
Similarly, the conical structure at infinity of $Y$ implies the existence 
of a large radius $R_0$ and of a homeomorphism $H :Y_{\geq R_0} \to Y_{R_0} 
\times [R_0,\infty)$ preserving the family of links $(Y_R)_{R\geq R_0}$, 
i.e. 
$$
R\geq R_0 \; \Longrightarrow \; H(Y_R) = Y_{R_0} \times R.
$$
Let $\cC_R^1,\ldots,
\cC_R^c$ be the connected components of $Y_\gR$. Thus
$$
Y^\infty := \cup_{i=1}^c Y_i^\infty \;\; {\rm where} \;\;
Y_i^\infty := (\cC_R^i)^\infty.
$$
Complementing Lemma \ref{lem:asympt-u1-u2}, we also find the following
useful
\begin{lemma}\label{lem:asympt-not-LNE}
Let $Y$ be a closed definable subset of $\Rq$, and let $\bx_0 \in Y$. 

(i) Let $\bx_0$ be a point of $Y$ such that the germ $(Y \setminus \bx_0,
\bx_0)$ has $c\geq 2$ connected components. Assume there exists $\bu \in S_{\bx_0}Y$
such that $\bu \in S_{\bx_0}Y_i$ for $i=1,2$. Then for each $r\leq r_0(\bx_0)$ 
there exist $C^1$ definable arcs $\by_i :[0,r] \to (Y_i)_\lr$ such that
$|\by_i(t) - \bx_0| = t$ for $i=1,2$ and 
$$
\lim_{t\to 0} \frac{d_\inn^Y(\by_1(t),\by_2(t))}{|\by_1(t) - \by_2(t)|} = \infty.
$$

(ii) Assume that the germ $(Y,\infty)$ has $c\geq 2$ connected components. 
Assume there exists $\bu \in Y^\infty$ such that $\bu \in Y_i^\infty$ for $i=1,2$. Then for
each $R\geq R_0$ there exist $C^1$ definable arcs $\by_i :[R,\infty) \to 
\cC_R^i$ such that $|\by_i(t)| = t$ for $i=1,2$ and 
$$
\lim_{t\to \infty} \frac{d_\inn^Y(\by_1(t),\by_2(t))}{|\by_1(t) - \by_2(t)|}
= \infty.
$$
\end{lemma}
\begin{proof}
We start with point (ii).
Since $\bu$ is a 
point of $Y_1^\infty \cap Y_2^\infty$, for $i=1,2$, there exist 
definable paths
$$
\by_i : [R,\infty) \to \cC_R^i, \;\; {\rm with} \;\; \lim_{t\to\infty}
\frac{\by_i(t)}{|\by_i(t)|} = \bu.
$$
We can further assume that $\by_i (t) = t (\bu + o(1))$ with 
$|\by_i(t)| = t$ for $i=1,2$.
Thus the function $t\to e(t) = |\by_1(t) - \by_2(t)|$ is such that
$$
e(t) = o(t) \;\; {\rm as} \;\; t \to \infty.
$$
For $t\geq R$, let $\gm_t:[0,1] \to Y$ be any rectifiable path connecting 
$\by_1(t)$ to $\by_2(t)$. Since $\cC_R^1\cap \cC_R^2$ is empty,
we obtain the next estimate  
$$
\lgt(\gm_t)\geq |\by_1(t) - R| + |\by_2(t) - R| = 2t - 2R, 
$$
from which we deduce the following one 
$$
\lim_{t\to\infty} \frac{d_\inn^Y(\by_1(t),\by_2(t))}{e(t)} = \infty.
$$

Point (i) is proved along the very same lines: Let $r\leq r_0 = r_0(\bx_0)$
so that $(Y_1)_\lr \cap (Y_2)_\lr = \{\bx_0\}$. 
Since $\bu$ lies in $S_{\bx_0} Y_1\cap S_{\bx_0 }Y_2$, for $i=1,2$,
there exist definable paths
$$
\by_i : [0,r] \to (Y_i)_\lr, \; t\mapsto \by_i(t) = \bx_0 + t\bu + o(t).
$$
We deduce that the function $t\mapsto e(t) = |\by_1(t) - \by_2(t)|$ 
satisfies the following estimate
$$
e(t) = o(t), \;\; {\rm as} \;\; t\to 0.
$$
For $t\leq r$, let $\gm_t:[0,1] \to Y$ be any rectifiable path connecting 
$\by_1(t)$ to $\by_2(t)$. We can assume that $2 r \leq  r_0$.
Since $(Y_1)_{\leq r_0} \cap (Y_2)_{\leq r_0}$ reduces 
to $\{\bx_0\}$, and since
$|\by_i(t)-\bx_0| = t$ once $t$ is 
close enough to $0$, we obtain the estimate  
$$
\lgt(\gm_t)\geq |\by_1(t) - \bx_0| + |\by_2(t) - \bx_0| = 2t, 
$$
from which we deduce the announced one 
$$
\lim_{t\to 0} \frac{d_\inn^Y(\by_1(t),\by_2(t))}{e(t)} = \infty. 
$$
\end{proof}
%
%
%
%
%
%
%
%
%
%
%
%
%
%
%
%
%
%
%
%
%
%
%
%
%
%
%
%
%
%
%
%
%
%
%
%
%
%
%
\section{LNE representatives}\label{section:L-RD}
%
%

%
%

The contents of this section are inspired by the metric notions in the definable setting from \cite{Kur1,KuOr,KuPa}. 
First, we recall the  essential result \cite[Theorem 1.3]{KuPa} 
\begin{theorem}\label{thm:KuPa}
There exists positive a constant $M = M(q)$ such that any definable subset
$A$ of $\Rq\times\Rp$ can be partitioned into a definable and finite union
$A = \cup_{i\in I} \cB^i$, such that for each $\bt\in\Rp$, every
subset~$\cB_\bt^i$ is LNE with constant $M$.
\end{theorem}
%
%
%
%
%

%
%

\bigskip
Let $X$ be a closed  definable subset of $\Rq$. Assume that $X$ is connected, 
contains the origin and that the germ $(X,\infty)$ is connected. We define
the following closed definable subsets of $\Rq\times\R$
\begin{eqnarray*}
\cB(X) & := & \{(\bx,t) \in X \times [0,\infty) : |\bx|\leq t\}, \\
\cU(X) & := & \{(\bx,t) \in X \times [0,\infty) : |\bx|\geq t\}.
\end{eqnarray*}
Consider the projection $\rho:\Rq \to [0,+\infty)$, 
$(\bx,t) \mapsto t$. Therefore 
$$
\cB(X)_t = X_{\leq t}, \;\; 
{\rm and} \;\; 
\cU(X)_t = X_{\geq t}.
$$
The following result is a variation of \cite[Th\'eor\`eme]{KuOr} for 
definable families with parameters.
\begin{proposition}\label{prop:defin-metrics-param}
Let $Z \in \{\cB(X),
\cU(X)\}$. Assume that $Z_t$ is connected or empty 
for each $t>0$. There exist a positive constant $M$ and a definable
family $(\Dlt_t)_t$ of distances over $(Z_t)_t$ such that for each 
positive $t$ we have
$$
\Dlt_t \; \leq \; d_\inn^{Z_t} \; \leq \; (1 + M) \Dlt_t.
$$
\end{proposition}
\begin{proof}
Theorem \ref{thm:KuPa} ensures that each $Z_t$ is 
a finite union of LNE sets $B_t^i$ with constant $C$.
Since each $(B_t^i)_t$, $i\in I$ is a definable family (the so-called $L$-cells), we reproduce the 
proof of \cite[Lemma 4]{KuOr} with $(1+M) = (1+C)^{q-1}$ for $Z_t$
and obtain that $(\Dlt_t)_t$ is a definable family by construction.
\end{proof}
\begin{remark}\label{rmk:inner-LNE}
Recall that $d_{Z_t} \leq \Dlt_t$. 
Since both families $(d_{Z_t})_t$ and $(\Dlt_t)_t$ are 
definable, the property of being LNE is definable for definable sets. In other words, by equivalence of metrics in Proposition~\ref{prop:defin-metrics-param}, we can
assume for the purpose of this paper that the family of inner distances 
$(d_\inn^{Z_t})_t$ is definable.
\end{remark}
The first interesting consequence of the Proposition~\ref{prop:defin-metrics-param} is the following
result about LNE representatives at $\infty$ and at $\bbo$.
\begin{corollary}\label{cor:repres-LNE-local}
(i) If $X$ is locally LNE at $\infty$, then there exist a positive radius 
$R_0$ and a positive constant $L$ such that $X_{\geq R}$ is LNE with LNE 
constant $L$ for radii $R \geq R_0$. 

\smallskip
(ii) If $X$ is locally LNE at $\bbo$, then there exist a positive radius 
$r_0$ and a positive constant $L$ such that $X_{\leq r}$ is LNE with
LNE constant $L$, for radii $0< r\leq r_0$.
\end{corollary}
\begin{proof}
Let us show (i). 
There exists a compact subset $K$ of $\Rq$, such that $X\setminus K$ is
LNE.
Denote 
$$
d^R := d_\inn^{X_{\geq R}}.
$$
Following Remark \ref{rmk:inner-LNE} we assume that  
the family of distances $(d^R)_R$ is definable.
Assume that $X\setminus K$ contains $X_{\geq x}$ for some positive $x$.
Let 
$$
D(R) := \sup \left\{\frac{d^R(\bx,\bx')}{|\bx - \bx'|} \; : \; 
\bx,\bx'\in X_{\geq R}, \, \bx\neq \bx' \right\} .
$$
The function $R\to D(R)$ is definable with values in $[0,\infty]$. The
first part of the statement we are looking for is equivalent to claiming
that for each $R \geq R_0$, we must find $D(R) < \infty$. 

Assume the statement is not true, by definability of function 
$R\to D(R)$, we can choose $x$ such that $D(R) = \infty$ for $R\in 
I_x := [x,\infty)$. Therefore for any definable function $\rho:[x,\infty) \to 
[0,\infty)$ such that $\rho(R) \to  \infty$ as $R \to \infty$, there exist 
definable arcs
$$
\aph,\beta :[x,\infty) \to X_{\geq x}  \;\; {\rm with} \;\; 
\aph(R),\beta(R) \in X_{\geq R}, \;\; \forall R \geq x
$$
such that 
$$
d^R(\aph(R),\beta(R)) \geq \rho(R) \cdot e(R), \;\; {\rm where} \;\;
e(R) := |\aph(R) - \beta(R)|.
$$
Consider the definable functions 
$A := |\aph|$ and $B:=|\beta|$. Following \eqref{eq:arc-infty}, we
write as $R$ goes to $\infty$
$$
\aph(R) = A(R)[\ba + o(1)] \;\; {\rm and} \;\;
\beta(R) = B(R)[\bb + o(1)], \;\; {\rm where} \;\;
\ba,\bb \in \bS^{q-1}.
$$
We can further assume that over $I_y=[y,+\infty)$ with $x\leq y$, the arcs 
$\aph,\beta$ are $C^1$, the functions $A:=|\aph|$, $B:=|\beta|$ and $\rho$ 
are $C^1$, strictly increasing, and with positive derivatives, and also such 
that $B\geq A$. Since $A:I_y \to I_{A(y)}$ 
and $B:I_y \to I_{B(y)}$ are $C^1$ definable diffeomorphisms, we deduce
$$
|\aph(A^{-1}(R))| =|\beta(B^{-1}(R))| =R.
$$ 
Following \cite[Corollary 1.3]{KuPa}, there exists a positive constant $C$
such that for every $R$ there exists a continuous definable path $\tht_R : 
[0,1] \to X_R$ such that $\tht_R(0) = \aph(A^{-1}(R))$, $\tht_R(1) = 
\beta(B^{-1}(R))$ 
and 
$$
\lgt(\tht_R) \; \leq \; C\cdot {\rm diam} (X_R),
$$
where the diameter is taken for the outer metric. 
Consider the rectifiable path $\gm_R$ with values in $X_\gR$, connecting 
$\aph(R)$ to $\aph(A^{-1}(R))$ along $\aph$, then reaching $\beta(B^{-1}(R))$ 
following 
$\tht_R$ and ending at $\beta(R)$ along $\beta$. The arguments used in the
proof of point (ii) of Lemma \ref{lem:asympt-u1-u2} would show the existence
of a constant $F$ such that $\lgt(\gm_R) \leq F\cdot e(R)$ for large radii 
$R$. Thus the following must hold true 
$$
\ba =\bb \;\; {\rm and} \;\; B = A + o(A).
$$
Therefore as $R\to\infty$, we deduce that
$$
e(R) = o(A(R)).
$$
\begin{claim}\label{claim:A(R)=R}
$A(R) = R(1 + o(1))$ for $R\gg 1$.
\end{claim}
\begin{proof}
Assume there exists $\lbd > 1$ such that $A(R) \geq \lbd R$ for $R$ large 
enough. Let $\gm_R :[0,1] \to X\setminus K$ 
be a rectifiable arc starting at $\aph(R)$ and ending at $\beta(R)$ such that
for $R$ large enough the following estimates hold
$$
\lgt(\gm_R) \leq (L+1) e(R) \ll \frac{\lbd-1}{2\lbd} A(R),
$$
implying that the arc $\gm_R$ is contained in $X_\gR$, which is 
prohibited.
\end{proof}
By hypothesis we cannot connect $\aph(R)$ and $\beta(R)$ within $X_{\geq R}$
to minimize their inner distance, 
therefore the following estimates are satisfied for large $R$
\begin{equation}\label{eq:close-R}
R \; < \;  A(R) \; \leq  \; B(R) \; < R + \rd_\inn^{X\setminus K}
(\aph(R),\beta(R)) \; \leq \; R + (L+1)e(R).
\end{equation}
In particular, $\aph(R)$ and $\beta(R)$ can be connected by a rectifiable 
curve contained in $X_{> R- (L+1)e(R)}$.
Consider the following definable function
$$
f: [R_0,\infty) \to \R, \;\; R\mapsto R - (L+1) e(R).
$$
Observe that $f (R)= R + o(R)$ for large $R$. Let $S(R) := f^{-1}(R)$.
Since $f$ is a diffeomorphism for large $R$, for any positive $\eta$, there 
exists $R_\eta \geq R_0$ such that the following estimate holds true 
over $[R_\eta,\infty)$
$$
|S(R) - (R+(L+1)e(R))|\leq \eta \cdot e(R)
$$
Since $S(R) - (L+1) e(S(R)) = R$, we deduce from Equation \eqref{eq:close-R} 
that we can connect $\aph(S(R))$ to $\beta(S(R))$ within $X_\gR$.
Following Claim \ref{claim:A(R)=R} and Estimate 
\eqref{eq:arc-infty-derivative}, for large $R$ we find that
$$
|\aph'|, |\beta'| = 1 + o(1).
$$
We can further assume that the functions $|\aph'|, |\beta'|$ are bounded by 
$1+ \eta$ over $[R_\eta,\infty)$. Thus 
$$
\lgt (\bc([R,S(R)]))  \leq  \int_R^{R + (L+1+\eta) e(R)} |\bc'| \leq (1+\eta) 
(L+1+\eta) e(R), \;\; 
{\rm for} \;\; \bc = \aph,\beta.
$$
Let $\dlt_R$ be a rectifiable curve connecting $\aph(S(R))$ to 
$\beta (S(R))$ contained in $X_\gR$ and such that
$$
\lgt(\dlt_R) \; \leq \; (L+1)e(S(R)).
$$
Let $\gm_R$ be the rectifiable curve starting at $\aph(R)$ going to
$\aph (S(R))$ along $\aph$, then reaching $\beta (S(R))$ following $\dlt_R$ 
and finally ending at $\beta(R)$ along $\beta$. Thus $\gm_R$ is contained 
on $X_\gR$. 
\begin{figure}[h]
	\centering
	\includegraphics[width=0.45\textwidth]{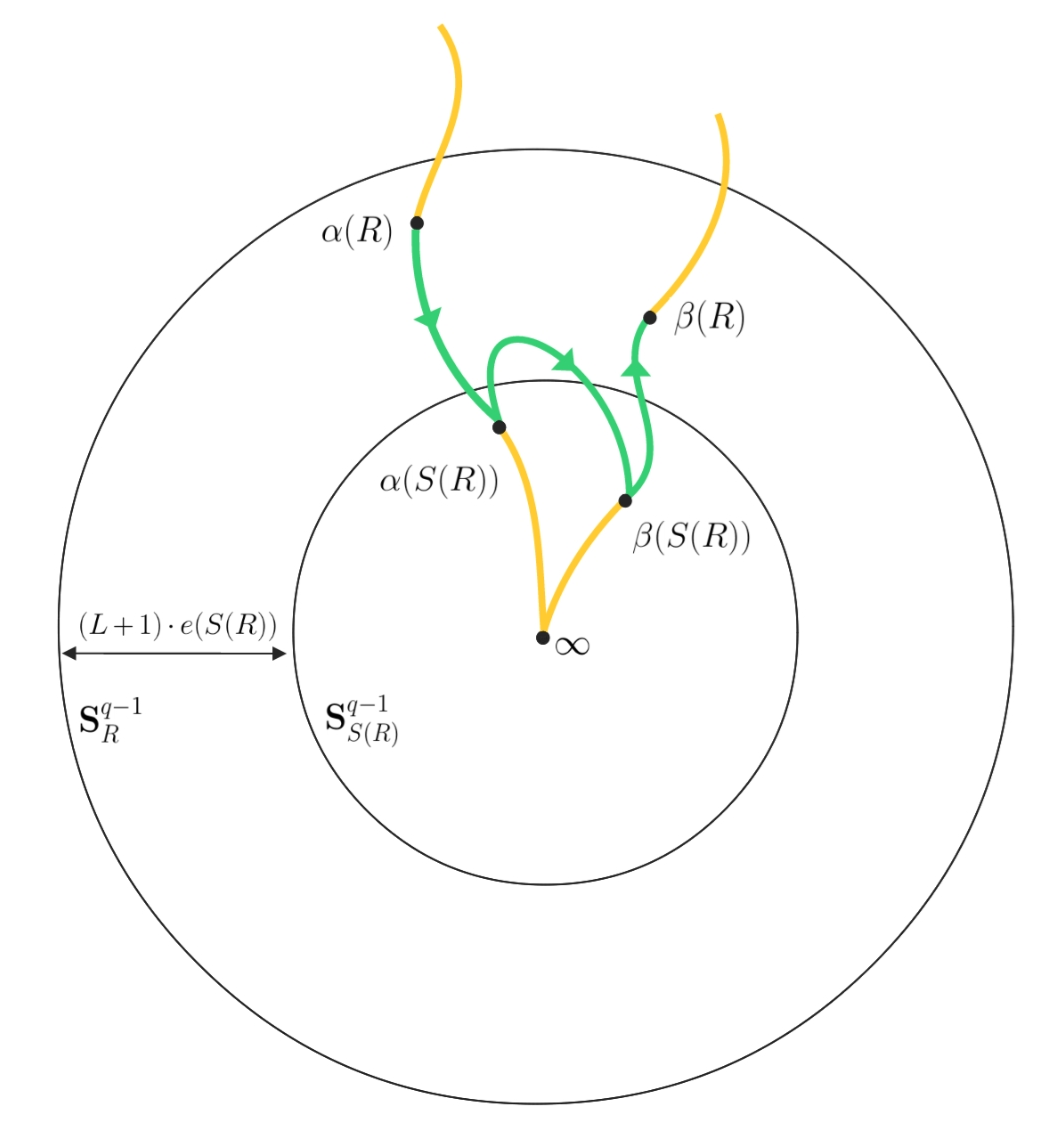}
	\caption{The green path $\gm_R$ is contained in $X_\gR$, and of length 
comparable to $|\aph(R) - \beta(R)|$.}
\end{figure}
%
%

\smallskip\noindent
Thus we deduce the following estimates
$$
\lgt(\gm_R) \; \leq \; (1+\eta)(L+1+\eta)e(R) + (L+1)e(S(R)) + 
(1+\eta)(L+1+\eta)e(R)
$$
for $R\geq R_\eta$ large enough. But the next estimate
$$
e(S(R)) = e(R) (1+o(1)) \;\; {\rm as} \;\; R \to \infty,
$$ 
produces a contradiction to the non-LNE hypothesis of $X_\gR$.

\medskip
We are left to show that we can find a LNE constant $M$ for 
each $X_\gR$ with $R\geq R_0$. For each $R$, let $M(R)$ be the infimum of the 
LNE constant of $X_\gR$. Since we can suppose that the family of distances
$(d^R)_R$ is definable, the function $R\to M(R)$ is non-decreasing, 
definable over $[R_0,\infty)$, thus continuous 
if $R_0$ is large enough. Either it is bounded or goes to $\infty$ as $R$ 
goes to $\infty$. But the previous part of the proof shows that it must be 
bounded.

\medskip\noindent
To obtain (ii), we will first assume that $X_{\lr}$ is not LNE for each 
positive radius $r\leq r_0$. The family of distances $(d_\inn^{X\lr})_r$
can be assumed definable. Thus there are two definable 
arcs $\aph,\beta :[0,r_0] \to X_{\leq r_0}$ such that
$|\aph(r)|,|\beta(r)| \leq r$ and such that
$$
\lim_{r\to 0}\frac{d_\inn^{X_\lr}(\aph(r),\beta(r))}{e(r)} = \infty \;\; 
{\rm where} \;\; 
e(r) = |\aph(r) - \beta(r)|. 
$$
From the arguments of point (i) of Lemma \ref{lem:asympt-u1-u2}, we 
deduce that $\bc(r) = r \bu + o(r)$, where $\bc =\aph,\beta$.
From here on, it is 
enough to follow the proof of case (i) with straightforward adaptations.
\end{proof}
\begin{remark}\label{rmk:repres-radius-like}
Instead of working with the 
family $(X_\lr)_{r\leq r_0}$ of LNE representatives
of $(X,\bbo)$, we could use definable radius like functions $\rho : 
(\Rq,\bbo) \to [0,\infty)$, as in \cite{Ngu,Val2}, to obtain 
that 
$$
X \cap\{|\bx|\leq \rho(r)\}$$
is LNE with LNE constant independent of $r \leq r_0^\rho$, 
since the proofs will be the same, substituting $\rho(r)$ for $r$.  

Definable radius-like functions $\rho : (\Rq,\infty)
\to [1,\infty)$ near infinity can be defined analogously to \cite{Ngu,Val2}, and we could
show as well that the definable family of representatives of the germ of $X$ at infinity
$$
X \cap\{|\bx|\geq \rho(R)\}
$$
consists of LNE sets with LNE constant independent of $R \geq R_0^\rho$.
\end{remark}
%
%
%
%
%
%
%
%
%
%
%
%
%
%
%
%
%
%
%
%
%
%
%
%
%
%
%
%
%
%
%
%
%
%
%
%
%
%
%
%
%
%
%
%

\section{Complements of representatives}\label{section:LNEaI}
The results developed here generalize those of our paper 
\cite[Sections 6,7]{CoGrMi1}.
The results a posteriori are obvious consequences of the main result, but they are needed beforehand for its proof.

\medskip
Let $X$ be a closed, definable and unbounded subset of $\Rq$. 
Let $R_0$ be a large enough radius so that the conical structure at infinity 
applies to $X_{\gR_0}$. Thus the $c\geq 1$ connected components of $X_\gR$ are connected cylinders $\cC_R^1,\ldots, \cC_R^c$, whenever $R\geq R_0$. Denote
$$
\cC^{i,\infty} := (\cC_R^i)^\infty.
$$
If $X_1^\infty,\ldots,X_s^\infty$ are connected components of $X^\infty$, then
 $s=c$ and  
 there exists a permutation $\ve$ of 
$\{1,\ldots,c\}$ such that $\cC^{i,\infty} = X_{\ve(i)}^\infty$ for 
each $i=1,\ldots,c$ (follows in particular from point (ii) of Lemma \ref{lem:asympt-not-LNE}).
\begin{lemma}\label{lem:Ci-LNE}
Let $X$ be a closed, connected definable and unbounded subset of $\Rq$. 
If $X$ is LNE, then each $\cC_R^i$ is locally LNE at $\infty$.
\end{lemma}
\begin{proof}
Assume that $\cC := \cC_R^1$ is not locally LNE at $\infty$. Let 
$$
d^X := d_\inn^X \;\; {\rm and} \;\; d^R := d_\inn^\cC. 
$$
By Proposition \ref{prop:defin-metrics-param}, we can assume that 
$d^X$ and $d^R$ are definable distances. Therefore there exist two definable 
arcs $\aph,\beta :[R,\infty)\to \cC$
such that  
$$
\lim_\infty|\aph| = \lim_\infty |\beta| = \lim_{t \to\infty} 
\frac{d^R(t)}{e(t)}  = \infty \;\; {\rm where} \;\; 
\left\{
\begin{array}{rcl}
d^R (t) &:= &d^R(\aph(t),\beta(t)) \\ 
e(t) & :=& |\aph(t) - \beta(t)|
\end{array}
\right.
$$
Let $A := |\aph|$ and $B := |\beta|$. We write again
$$
\aph(t) = A(t)(\ba + o(1)) \;\; {\rm and} \;\; \beta(t) = B(t)(\bb + 
o(1)).
$$
Point (ii) of Lemma \ref{lem:asympt-u1-u2} provides the following 
$$
\ba = \bb \;\; {\rm and} \;\; B = A(1+o(1)).
$$
Since we can assume that both $A$ and $B$ are strictly increasing to 
$\infty$ as $t$ goes to $\infty$, point (ii) of Lemma \ref{lem:asympt-u1-u2} 
implies the following 
$$
\ba = \bb, \;\; B = A(1+o(1)), \;\; {\rm and} \; {\rm thus}
\;\; e = o(A). 
$$
Since $X$ is LNE we find the following estimate
$$
\lim_{t\to\infty} \frac{d^X(\aph(t),\beta(t))}{A(t)} = 0.
$$
Therefore any rectifiable curve $\gm_t$ in $X$ connecting $\aph(t)$ and
$\beta(t)$ and such that
$$
\lgt(\gm_t) \leq 2 d^X(\aph(t),\beta(t))
$$
must be contained in $\cC$ once $t$ is large enough, contradicting 
the initial hypothesis that $\cC$ was not locally LNE at $\infty$.
\end{proof}
\begin{proposition}\label{prop:X-LNE-repr}
Let $X$ be a closed connected definable subset of $\Rq$. 

(i) If $X$ is LNE, then there exist a positive radius $R_0$ and a positive 
constant $L$ such that $X_\lR$ and each connected components of $X_\gR$ are 
LNE with LNE constant $L$ whenever $R\geq R_0$. 

(ii) Assume there exists $R_0$ such that $X_\lR$ is LNE, each connected
component $\cC_R^i$ of $X_\gR$, $i=1,\ldots,c,$ is LNE for $R\geq 
R_0$, and $(\cC_R^i)^\infty \cap (\cC_R^j)^\infty = \emptyset$ if $i\neq j$. 
Then $X$ is LNE
\end{proposition}
\begin{proof}
We start with point (i). 
Lemma \ref{lem:Ci-LNE} and point (i) of Corollary \ref{cor:repres-LNE-local} 
guarantee the statement is true for the connected components of 
$X_\gR$ whence $R\geq R_0$. Since $X$ is connected, we take $R_0$ large 
enough so that $X_\lR$ is connected once $R\geq R_0$.

\medskip
Let $d^R$ be the inner metric $d_\inn^{X_\lR}$ of $X_\lR$. 
Remark \ref{rmk:inner-LNE} guarantees that we can assume that the family of 
distances $(d^R)_R$ is definable. Thus, either $X_\lR$ is LNE for $R\geq 
R_1$ or $X_\lR$ is not LNE for $R\geq R_1$, for some radius $R_1\geq R_0$. 

\medskip\noindent
$\bullet$ Assume that the family $(X_\lR)_R$ is not LNE as $R \to \infty$. 
Therefore
there exist definable arcs $\aph,\beta:[R_0,\infty) \to X$ such that 
$A(R),B(R) \leq R$, where $A:=|\aph|$ and $B=|\beta|$, and
$$
\lim_{R\to\infty}\frac{d^R(\aph(R),\beta(R))}{e(R)} = \infty \;\;
{\rm where} \;\; e := |\aph - \beta|.
$$
Writing again as $R\to \infty$
$$
\aph = A(\ba+o(1)) \;\; {\rm and} \;\; \beta = B(\bb + o(1))
$$
we deduce again, using (ii) of
Lemma \ref{lem:asympt-u1-u2}, when $R\to\infty$ that
$$
\ba = \bb, \;\; B =  A(1+o(1)), \;\; {\rm and} \; {\rm thus}
\;\; e = o(A). 
$$
We can assume that $R$ is large enough so that $A\leq B$, and $A,B$ are $C^1$
and strictly increasing. Thus $X_{\geq A(R)}$ is LNE, for $R \geq R_0^A := 
A^{-1}(R_0)$ if $R_0$ is large enough. Let $L$ be a LNE constant common 
to all $X_{\gR}$ and $X_{\geq A(R)}$ for $R\geq R_0^A$. 
There exists a rectifiable path $\gm_R$ taking values in $X_{\geq A(R)}$
connecting $\aph(R)$ to $\beta(R)$ such
that 
$$
\lgt(\gm_R) \leq (L+1)e(R). 
$$
If $\lim_{R\to\infty}\frac{A(R)}{R} < 1$, we deduce that for $R$ large enough
that $\gm_R$ is contained in $X_\lR$, which is a contradiction. Therefore
$$
A(R) = R + o(R).
$$
Thus $A',B' \to 1$ as $R\to\infty$. Given
$\eta >0$, for $R \geq R_\eta$ the following estimates hold true 
$$
|A' - 1|, \; |B'-1| \; \leq \; \eta.
$$
Since $\gm_R$ cannot be contained in $X_\lR$, we deduce that 
$$
R \; \geq \; B(R) \; \geq \; A(R) \; \geq \;R - (L+1) e(R).
$$
The function $f:[R,\infty) \to \R$, $R\to R + (L+1)e(R)$ is definable, 
of the form $R + o(R)$, so that $\lim_\infty f' = 1$. 
Let $S := f^{-1}$. Given some positive $\eta$, we can further assume that
$$
|S(R) - (R - (L+1)e(R))| \; \leq \; \eta e(R)
$$
whenever $R\geq R_\eta$. We also observe that
$$
\lgt(\bc([S(R),R])) \; \leq 
\; L_\eta \cdot e(R),
\;\; {\rm for} \;\; \bc =\aph,\beta.
$$
where $L_\eta := (1+\eta) (L+1+\eta)$.
We deduce that there exists $\gm_S :[0,1] \to X_\lR$, connecting $\aph(S(R))$
to $\beta (S(R))$ such that
$$
\lgt(\gm_S) \leq (L+1/2)e(S(R))
$$
so that $\gm_S$ is contained in $X_\lR$. Let $\gm_R$ be a rectifiable path 
contained in $X_\lR$, going from $\aph(R)$ to $\aph(S(R))$ along $\aph$, then
going to $\beta(S(R))$ following $\gm_S$ and then reaching $\beta(R)$ 
along $\beta$. 

\smallskip\noindent
Thus we find
$$
\lgt(\gm_R) \; \leq \; L_\eta \cdot e(R) + (L+1/2)e(S(R)) + L_\eta \cdot e(R)
$$
obtaining a contradiction since $e(S(R)) = e(R) (1+o(1))$.

\medskip 
Proving that the definable family $(X_\lR)_{R\geq R_0}$ admits a 
uniform LNE constant is done exactly as at the end of the proof
of point (i) of Corollary \ref{cor:repres-LNE-local}.

\medskip\noindent
$\bullet$ Assume (ii). 
Thus $X_\lR$ and the connected components of $X_\gR$ are LNE for 
$R\geq R_0$. Let 
$$
D(R) := \sup \left\{\frac{d_\inn^X(\bx,\bx')}{|\bx - \bx'|}, \;\; 
{\rm where} \;\;
\bx \in X_\lR, \; \bx'\in X_\gR, \; \bx \neq  \bx' \right\}.
$$
Since $(\cC_R^i)^\infty \cap (\cC_R^j)^\infty = \emptyset$, point (ii) of 
Lemma \ref{lem:asympt-not-LNE} implies that $X$ is LNE 
if and only if $D(R) <\infty$.
If $D(R) =\infty$, then there are sequences $(\bx_n)_n$ of $X_\lR$ 
and $(\bx_n')_n$ of $X_\gR$ such that 
$$
\frac{d_\inn^X(\bx_n,\bx_n')}{|\bx_n - \bx_n'|} \to \infty
$$
Up to passing to subsequences we deduce that $\bx_n,\bx_n'$ 
converge to $\by \in X_R$. Therefore there exists $S \geq R$ so that both 
sequences $(\bx_n)_n$ and $(\bx_n')_n$ are contained in $X_{\leq S}$, 
contradicting the property of the pair of sequences.
\end{proof}

\smallskip
To finish this section, we present an analogue of Lemma 
\ref{lem:Ci-LNE} at the origin (see the 
sub-analytic case  in \cite{MeSa} and the general case in \cite{Ngu}).

\smallskip
Let $X$ be a closed definable subset of $\Rq$ containing the point $\bx_0$.
Let $r_0$ be a small enough radius for which 
$X \cap \bB^q(\bx_0,r_0)$ is homeomorphic to the cone over $X \cap 
\bS^{q-1}(\bx_0,r_0)$ with vertex $\bx_0$.
Let $Y_1,\ldots,Y_c,$ be the closures of the connected components 
of $X\cap \bB^q(\bx_0,r_0)\setminus \bx_0$. Each $Y_i$ is a closed connected 
definable subset containing the origin. We recall that 
$$
S_{\bx_0} X := \cup_{i=1}^c S_{\bx_0} Y_i.
$$
%
%
\begin{lemma}\label{lem:Yi-LNE}
Let $X$ be a closed definable subset of $\Rq$ containing the origin.
If $X$ is locally LNE at $\bx_0$, then 

\smallskip
1) $1\leq i<i'\leq c \; \Longrightarrow S_{\bx_0} Y_i \cap S_{\bx_0} Y_{i'} 
=\emptyset$.

\smallskip
2) $Y_i$ is locally LNE at $\bx_0$ for each $i=1,\ldots,s$.
\end{lemma}
\begin{proof}
We assume $\bx_0 = \bbo$. 
Point 1) follows from point (i) of Lemma \ref{lem:asympt-not-LNE}. 
Point 2) will follow from the scheme of the proof of Lemma 
\ref{lem:Ci-LNE}, adapting the arguments from the situation at $\infty$
to that at $\bbo$.
\end{proof}
\begin{remark}\label{rmk:raidus-like-2}
Similarly to Remark \ref{rmk:repres-radius-like}, the statements of this section
have all analogues with a radius like function nearby $\bbo$ (or nearby $
\infty$).
\end{remark}
%
%
%
%
%
%
%
%
%
%
%
%
%
%
%
%
%
%
%
%
%
%
%
%
%
%
%
%
%
%
%
%
%
%
%
%
%
%
%
%
%
%
%
%
\section{LNE property under inversion}\label{section:LLNEaIvsO}

Let $\ioq:\Rq\setminus\bbo \to \Rq\setminus\bbo$ be the Euclidean inversion
defined as
$$
\bx \mapsto \ioq(\bx) := \frac{\bx}{|\bx|^2}.
$$
It is a  $C^\infty$ diffeomorphism definable in any
o-minimal structure.
\begin{lemma}\label{lem:LNE-inversion}
Let $X$ be a closed definable subset of $\Rq$ containing $\bbo$ and such
that the germ $(X\setminus\bbo,\bbo)$ is connected. The subset $X$ is locally 
LNE at $\bbo$ if and only if $\ioq(X\setminus \bbo)$ is locally LNE at 
$\infty$.
\end{lemma}
\begin{proof}
Again, following Remark \ref{rmk:inner-LNE} the property
of being locally LNE at $\omg$, with $\omg = \infty, \bbo$, is definable.
For the proof, we write 
$$
\wtX := \ioq(X\setminus \bbo).
$$

\medskip\noindent
$\bullet$ \em Assume that $X$ is locally LNE at $\bbo$. \em 

\smallskip 
According to point (ii) of Corollary 
\ref{cor:repres-LNE-local}, there exists a positive radius  
$r_0$ such that $X_\lr$ is LNE for any positive radius $r\leq r_0$ with LNE 
constant $L$ independent on $r$. 

\medskip
Let $R_0 := \frac{1}{r_0}$ 
and $R := \frac{1}{r}$. 

\smallskip
Suppose that $\wtX$ is not locally LNE at $\infty$. Therefore there exist  
definable arcs $\wtap,\wtbe :[R_0,\infty) \to \wtX_{\geq R_0}$ such that
$\wtap,\wtbe \to \infty$ as $R\to \infty$ and 
$$
\lim_{R\to\infty} \frac{d_\inn^{\wtX\setminus K}(\wtap(R),\wtbe(R))}{E(R)} = 
\infty, \;\; {\rm where} \;\; E(R) :=|\wtap(R) - \wtbe(R)|,
$$     
where $K$ is a compact of $\Rq$ such that $\wtX\setminus K$ is a LNE.
Denoting $A:=|\wtap|$ and $B:=|\wtbe|$, 
let us write again
$$
\wtap = A(\ba + o(1)) \;\; {\rm and} \;\; \wtbe = B(\bb + o(1)).
$$
Let $\aph,\beta : ]0,r_0]\to X_{\leq r_0}\setminus \bbo$ be the definable 
arcs defined as
$$
\aph := \iota_q \circ \wtap \circ \iota_1 = a(\ba + o(1)) \;\; {\rm and} \;\;
\beta := \iota_q \circ \wtbe \circ \iota_1 = b(\bb + o(1)).
$$
Let $e(r) := |\aph(r) - \beta(r)|$. Using \eqref{eq:law-cosine}, we get 
\begin{equation}\label{eq:E-e}
E(R) = A(R)\cdot B(R)\cdot e(r) \;\; {\rm with} \;\; r = \frac{1}{R}.
\end{equation}
First, point (ii) of Lemma \ref{lem:asympt-u1-u2} yields
$$
\ba = \bb \;\; {\rm and} \;\; B = A + o(A) \;\; {\rm as} \; R\to \infty.
$$
Thus we also find that $b = a + o(a)$ as $r\to 0$ and 
$E(R) = A(R)^2(1+o(1)) e(r)$. Therefore there exists a rectifiable curve
$\gm_r$ connecting $\aph(r)$ to $\beta(r)$ within $X_{\leq a(r)}$
since $a \geq b$ and such that
$$
\lgt(\gm_r) \; \leq \; (L+1) e(r)
$$
since $X_{\leq t}$ is LNE with LNE constant $L$ for any $t\leq r_0$. 
Furthermore, since
$e(r) = o(a(r))$ we can assume that $r$ is small enough so that
$$
\min |\gm_r| \geq \frac{a(r)}{2} = \frac{1}{2A(R)}. 
$$
Let $\wtgm_R := \iota_q (\gm_r)$ connecting $\wtap(R)$ to $\wtbe(R)$. 
Assuming that $\gm_r$ is parametrized over $[0,1]$, we 
deduce that 
$$
\lgt(\wtgm_R) = \int_0^1 \frac{|\gm_r'(t)|}{|\gm_r (t)|^2} \rd t \; \leq
\; 4 A(R)^2 \lgt(\gm_r) \; \leq 4 (L+1) A(R)^2 e(r).
$$
Since $B = A +o(A)$, the estimate of $\lgt(\wtgm_R)$ and Equation
\eqref{eq:E-e} produce a contradiction. 

\medskip\noindent
$\bullet$ \em Assume that $\wtX$ is locally LNE at $\infty$. \em

\smallskip
Suppose that $X$ is not locally LNE at $\bbo$. Therefore there exists  
definable arcs $\aph,\beta : [0,r_0] \to X_{\leq r_0}$ such that 
$\aph(0) = \beta(0) = \bbo$ and 
$$
\lim_{r\to 0} \frac{d_\inn^{X_{\leq r_0}}(\aph(r),\beta(r))}{e(r)} = \infty,
\;\; {\rm where} \;\; e(r) = |\aph(r) - \beta(r)|.
$$
Let $a := |\aph|$ and $b := |\beta|$ so that we can write
$$
\aph = a (\ba + o(1)) \;\; {\rm and} \;\; \beta = b (\bb + o(1)).
$$
Point (i) of Lemma \ref{lem:asympt-u1-u2} implies that as $r$ goes to $0$
$$
\ba = \bb, \;\; b = a + o(a) \;\; {\rm and} \; {\rm thus} \;\; 
e(r) = o(a(r)).
$$
Working with $\wtap := \iota_q\circ\aph$ and $\wtbe := \iota_q\circ\beta$,
and denoting $E := |\wtap - \wtbe|$, we conclude as in the last part of the
previous case working with $\wtap,\wtbe$ in $X_{\geq R_0}$ (which is LNE)
and using the identity   
$$
e(r) = a(r)b(r) E(R) = a(r)^2(1+ o(1)) E(R) \;\; {\rm for} \;\; R = \
\frac{1}{r}.
$$
when $r\to 0$ to produce a contradiction, as was done in the previous case.
\end{proof}
%
%

%
%
%
%
%
%
%
%
%
%
%
%
%
%
%
%
%
%
%
%
%
%
%
%
%
%
%
%
%
%
%
%
%
%
%
%
%
%
%
%
%
%
%
%
%
\section{Main result}\label{section:MR}
Let $N_q := (0,\ldots,0,1) \in \bS^q$ be the north pole of $\bS^q$.
Let 
$$
\sgm_q: \Rq \to \bS^q\setminus N_q, \;\; \bx \to \left(\frac{2\bx}{1+R^2},
\frac{R^2-1}{R^2+1}\right)\in\Rq\times[-1,1[
$$
where $R := |\bx|$, be the inverse of the stereographic projection 
$\bS^q\setminus N_q \to \Rq$.
For convenience, we consider the standard embedding of the sphere $\bS^q$  into $\Rq$ and thus the standard inner distance $d_\inn^{\bS^q}$ on the sphere is equivalent to the euclidean distance $d_{\bS^q}$ in~$\Rq$ restricted to the sphere
$$
d_{\bS^q} \; \leq \; d_\inn^{\bS^q} \; \leq \; \frac{\pi}{2} d_{\bS^q}.
$$
Thus a subset of $\bS^q$ is LNE in $\bS^q$ if and only if it is LNE  as a subset
of $\R^{q+1}$, for more generality compare \cite[Proposition 1.9]{CoGrMi2}.

\begin{lemma}\label{lem:main}
Let $X$ be a closed definable subset of $\Rq$ such that the germ $(X,\infty)$
is connected. The subset $X$ is locally LNE at $\infty$ if and only if 
$\sgm_q(X)$ is locally LNE at $N_q$ in $\bS^q$ if and only if 
$\sgm_q(X)\cup N_q$ is locally LNE at $N_q$ in $\bS^q$.
\end{lemma}
\begin{proof}
Consider the following mapping $\phi : \bB_\frac{1}{2}^q \to \bS^q$ defined
as 
$$
\by \to \left(\frac{1}{1+r^2}\cdot 2\by,\frac{1-r^2}{1+r^2}\right)
$$
where $r =|\by|$. It is a $C^\infty$ and semi-algebraic diffeomorphism
onto its image. In particular it induces a bi-Lipschitz
homeomorphism germ
$$
((\Rq,\bbo),|-|) \to ((\bS^q,N_q),d_{\bS^q}).
$$ 
We also check that
$$
|\bx| \geq 2 \; \Longrightarrow \; \phi\circ\ioq(\bx) = \sgm_q(\bx).
$$
Lemma \ref{lem:LNE-inversion} implies that $X$ is locally LNE at $\infty$
if and
only if $\ioq(X)$ is locally LNE at $\bbo$, if and only if $\sgm_q(X)$
is locally LNE at $N_q$ in $\bS^q$. The last equivalence is obvious.
\end{proof}
We can now prove now the main result:
\begin{theorem}\label{thm:main}
A closed connected definable subset $X$ of $\Rq$ is LNE in $\Rq$ if and only if $\sgm_q(X)\cup N_q$ is LNE in $\bS^q$.
\end{theorem}
\begin{proof}

If $X$ is compact the result is immediate since $\sgm_q$ is bi-Lipschitz over 
any compact subset of $\Rq$. 

\medskip\noindent
$\bullet$ \em  Assume $X$ is unbounded and LNE. \em Let 
$$
Z:= \sgm_q (X) \cup N_q = \clos(\sgm_q(X)).
$$ 
Since $X$ is LNE, and since the inversion is a $C^\infty$ diffeomorphism, 
the subset $Z$ is locally LNE at each point of $Z\setminus N_q$.
Let $R$ large enough so that 
$X_\gR$ is a union of finitely many closed connected cylinders 
$\cC_R^1,\ldots,\cC_R^c$. Let 
$$
Y_R^i := \sgm_q(\cC_R^i) \cup N_q = \clos(\sgm_q(\cC_R^i)).
$$
Combining Lemma \ref{lem:Ci-LNE}, Lemma \ref{lem:main} with point (ii) of 
Corollary \ref{cor:repres-LNE-local}, we can assume that $\cC_R^i$ is LNE 
and so is $Y_R^i$. For each $i=1,\ldots,c$ , we define
$$
C_i := (\cC_R^i)^\infty \;\; {\rm and} \;\;
S_i := S_N Y_R^i.
$$
First, for each $i=1,\ldots,c$, it is straightforward to check that 
$$
S_\bbo (\clos(\ioq(\cC_R^i))) = C_i. 
$$
Second using the application $\phi$ 
introduced in the proof of Lemma \ref{lem:main} we check
that, 
$$
D_\bbo \phi \cdot (\wh{C_i}^+) = \wh{S_i}^+ \subset \Rq \times 0 = 
T_N\bS^q. 
$$
In particular we deduce that $S_i \cap S_{i'}$ is empty whenever
$i,i'$ are distinct. 
\begin{claim}\label{claim:Z-LNE-N}
$Z$ is locally LNE at $N$ in $\bS^q$.
\end{claim}
\begin{proof}
For each $i=1,\ldots,c$, note that
$$
S_i \subset \bS^{q-1}\times 0 \subset T_N\bS^q = \R^q\times 0.
$$ 
To simplify notations, let
$$
Z(R) := \clos(\sgm_q(X_\gR))\;\;  {\rm and} \;\;
Z_i := Y_R^i. 
$$  
Let $\dist$ be the Euclidean distance in $\Rq\times\R$. Then 
$$
\dlt := \min_{1\leq i<j \leq c} \dist(S_i,S_j) 
$$
is positive. Thus we can assume that $R$ is large enough so that 
$$
\min_{1\leq i < j \leq c} \left\{ 
\dist \left( \frac{\bz_i-N_q}{|\bz_i-N_q|},\frac{\bz_j-N_q}{|\bz_j-N_q|} \right), 
\; \bz_i \in Z_i\setminus N_q, \;\bz_j \in Z_j\setminus N_q \right\} \; \geq \; 
\frac{\dlt}{2}.
$$
Therefore there exists $2\tht_Z \in ]0,\pi]$ such that for all $\bz_i 
\in Z_i\setminus N_q$ and $\bz_j \in Z_j \setminus N_q$, with $1\leq i<j\leq c$,
the non-oriented angle (in $[0,\pi]$) between $\bz_i-N_q$ and $\bz_j - N_q$
is larger than or equal to $2\tht_Z$.

Since $Z(R)$ is definable and satisfies to the conic structure theorem 
at $N_q$, when $R_0$ is large enough, by \cite[Corollary 3]{KuPa}
there exists a positive constant $D$, depending only on $R_0$, such that 
any points $\bz$ of $Z(R)$ can be joined to $N_q$ by a definable arc 
of length at most $D|\bz - N_q|$.

Let $\bz_i \in Y_i$ for $i=1,2$. Let $r_i := |\bz_i -N_q|$. 
To connect $\bz_1$ and $\bz_2$ in $Z(R)$ it is necessary to go through
$N$. Therefore
$$
d_\inn^{Z(R)}(\bz_1,\bz_2) \leq  D (r_1 + r_2).
$$
Using again the law of cosines \eqref{eq:law-cosine}, we observe that
$$
|\bz_1 - \bz_2| 
\; \geq \; \frac{\sin\tht_Z}{D} \cdot d_\inn^{Z(R)}(\bz_1,\bz_2). 
$$
Since each $Z_i$ is LNE in $\bS^q$, we conclude that $Z(R)$ is LNE
in $\bS^q$.
\end{proof}
$Z$ is compact, connected, and locally LNE at each of its point, 
thus is LNE in $\bS^q$ by Lemma \ref{lem:compact-LNE}.

\medskip\noindent
$\bullet$ \em  Assume $Z := \clos(\sgm_q(X))$ is LNE in $\bS^q$. \em  

Let $\by = (\by',y)$ be Euclidean coordinates $\R^{q+1} =\Rq\times\R$.

Let $\pi :\bS^q \to \Rq = T_N\bS^q = \Rq \times 0 \subset \R^{q+1} = 
T_N\R^{q+1}$, be the orthogonal projection. The image of $\pi(N)$ will be 
$\bbo$. Let $h \in (0,1)$ be given. 
The restriction of $\pi$ to $\bS^q \cap \{y 
\geq h\}$ is definable and bi-Lipschitz on its image. Thus, there exists a 
positive radius $R_0$ such that for each $R\geq R_0$ the subset 
$$
Z(R) := Z\cap \left\{y\geq \frac{R^2-1}{R^2+1}\right\} = 
\clos(\sgm_q(X_\gR))
$$ 
is LNE in $\bS^q$. If $R_0$ is large enough, then $Z(R)$ is connected for
any $R$. Let $Y_R^1,\ldots,Y_R^c,$ be the 
closure of the connected components of $Z(R)\setminus N_q$. Each of 
which is LNE by point (ii) of Corollary \ref{cor:repres-LNE-local}.

\medskip
Lemma \ref{lem:main} implies that each $\cC_R^i := \sgm_q (Y_R^i \setminus 
N_q)$ is locally LNE at $\infty$. Therefore up to, taking a larger $R_0$, we 
deduce that $X_\lR$ is connected and, by point (i) of Corollary 
\ref{cor:repres-LNE-local}, each connected component $\cC_R^i$ of $X_\gR$ 
is LNE, with a LNE constant uniform in $R\geq R_0$ and $i$. To conclude it
is enough to show that the definable family $(\sgm_q (X_\lR))_{R\geq R_0}$
is LNE in $\bS^q$. Since $Z$ is LNE in $\bS^q$, adapting the part of the 
proof of point (i) of Proposition \ref{prop:X-LNE-repr} about the definable 
family $(X_\lR)_{R\geq R_0}$  
being LNE, we show that $(\sgm_q(X_\lR))_{R\geq R_0}$ is LNE in $\bS^q$. 
\end{proof}
%
%
%
%
%
%
%
%
%
%
%
%
%
%
%
%
%
%
%
%
%
%
%
%
%
%
%
%
%
The following description of the LNE property of  the tangent cone is a straightforward consequence of Lemma~\ref{lem:LNE-inversion} combined with result of main result of \cite{FeSa1} (see also 
	\cite[Theorem 2.5.19]{BuBuIv} for a direct proof). 
	\begin{corollary}\label{cor:LNE-inversion}
		Let $X$ be a closed definable subset of $\Rq$ and let $\cC_1,\ldots,\cC_c$ be 
		the connected components of the germ $(X,\infty)$. 		
		If $\cC_1,\ldots,\cC_c$, are locally LNE at $\infty$ and  $\cC_i^\infty 
		\cap \cC_j^\infty = \emptyset$ for $1\leq i < j \leq c$, then 
		the tangent cone  $\wh{X^\infty}^+$ of $X$ at $\infty$ is LNE. 
	\end{corollary}
	
	Similarly, applying Lemma~\ref{lem:LNE-inversion} to 
	the main result of \cite{MeSa} in the 
	sub-analytic case and \cite{Ngu} in the general case, we get the LNE property of links at infinity. 
	
	\begin{corollary}\label{cor:LNE-inversionLINK}
		Let $X$ be a closed definable subset of $\Rq$. 
		The subset $X$ is locally LNE at 
		$\infty$ if and only if  there exist a large radius $R_0$ and a positive 
		constant $L$, such 
		that the link $X_R$ is LNE with LNE constant $L$, for each $R\geq R_0$.  
	\end{corollary}
From the main result we obtain the explicit gluing property.
	
	\begin{corollary}\label{cor:main}
		A connected closed definable subset $X$ of $\Rq$ is LNE if and only if  it is locally LNE at each of its points, each connected component 
		$\cC_1,\ldots,\cC_c$, of the germ $(X,\infty)$ is locally LNE at $\infty$, and
		$\cC_i^\infty\cap\cC_j^\infty = \emptyset$ for $1\leq i < j \leq c$.
	\end{corollary}

%
%
\section{LNE models of definable sets}\label{section:LNEmodels}

We can apply our main theorem to obtain a classification result extending the main result of \cite{BiMo} for compact semi-algebraic sets to all closed definable sets:
\begin{theorem}\label{prop:LNE-model}
	Let $X$ be a closed connected definable subset of $\Rp$. There exists a definable
	bi-Lipschitz homeomorphism $(X,d_\inn^X) \to (X',d_\inn^{X'})$, where  $X'$ is a LNE closed definable subset  of $\Rq$.
\end{theorem}

Proof of Theorem~\ref{prop:LNE-model} is presented at the end of this section. Note that 
the proof of main theorem of \cite{BiMo}, using the results of \cite[Section 1]{KuPa},  adapts 
easily to the setting of subsets definable in any o-minimal structure expanding the real numbers 
field instead of the semi-algebraic one.

\medskip
Let $X$ be a subset of an Euclidean space $\Rp$, and let $d = d_\inn^X$ be the inner metric
on $X$.

\em A Lipschitz curve in X \em is a Lipschitz mapping $\gm:[a,b]\to (X,d)$. \em The speed of $\gm$ at $t\in (a,b)$ \em is defined
as 
$$
\lim_{h\to 0}\frac{d(\gm(t+h),\gm(t))}{|h|}
$$
when the limit exists. 
Per \cite[Theorem 3.6]{Haj} the speed function  of
a Lipschitz curve $\gm:[a,b] \to (X,d)$  is defined almost everywhere  and equal
$
|\dot\gm(t)| 
$
for almost every $t$. 
Its length is given as
$$
\ell(\gm) = \int_a^b |\dot\gm(t)| \rd t,  
$$
Every rectifiable curve can be reparametrized by arc-length. Since the length of a curve is independent on taking outer or inner metric on $X$, arc-length parametrization is Lipschitz with respect to both metrics.

\medskip
Let $X'$ be another subset of an Euclidean space $\Rq$ and let $d' = d_\inn^{X'}$ be the inner metric on $X'$. Assume that there exists a Lipschitz mapping 
$f: (X,d) \to (X',d')$ with Lipschitz constant $L_f$. If $\gm:[a,b] \to (X,d)$ is a Lipschitz curve, then the curve $f\circ\gm:[a,b] \to (X',d')$ is also Lipschitz and 
such that for almost all $t$ the following  holds
$$
|\dot{\overgroup{f\circ\gm}} (t)| = |\dot{\overgroup{f\circ\gm}}|(t) 
\leq  L_f 
|\dot\gm(t)|.
$$

Recall that by \cite{GrOl} a mapping $\phi:(X,\bbo)\to (X',\bbo)$ is outer 
bi-Lipschitz between two closed germs of~$\Rp$ and $\Rq$ respectively if and 
only if
	$\ioq\circ\phi\circ\iop:(\iop(X\setminus\bbo),\infty)\to 
(\ioq(X'\setminus\bbo),\infty)$
is also outer bi-Lipschitz. The following Lemma~\ref{lem:inner-inversion} is 
a counterpart for the inner metrics, note that we do not assume that the sets 
are definable.
\begin{lemma}\label{lem:inner-inversion}
	Let $X\subset\Rp$ and $X'\subset\Rq$ be closed subsets of Euclidean spaces, both containing 
	 the origin and such that germ at $\bbo$ of each set is connected after removing the origin. Denote
	 $\wt{X} := \clos(\iop^{-1}(X\setminus \bbo))$ and $\wt{X'}:=\clos(\ioq({X'}\setminus \bbo))$.	
	 
	 Assume there exists  a positive constant $C$ such that
	 $$
	 d_\inn^{X}(\bx,\bbo) \leq C \,|\bx|, \; \forall \bx\in X, \;\; {\rm and} \;\; d_\inn^{X'}(\bx',\bbo) \leq C \,|\bx'|, \; \forall\bx'\in X',
	 $$
	 
	If  $\phi:(X,d_\inn^X) \to (X',d_\inn^{X'})$ is an inner bi-Lipschitz homeomorphism with	$\phi(\bbo) = \bbo$, 
	then
	the homeomorphism 
	$$
	\wtphi := \ioq\circ\phi\circ\iop^{-1} :\wtX \to \wt{X'},
	$$ 
	with $\wtphi(\bbo)=\bbo$ if $X$ is unbounded, is also inner bi-Lipschitz.
\end{lemma}
\begin{proof}
We need only show that $\wtphi$ is inner Lipschitz, since the same argument 
goes for $\wtphi^{-1}$. It suffices to show there exists a constant $N$ such 
that, for any {unit speed} rectifiable curve $\wt{\gamma}$ connecting $x$ 
with $y$ in $\wt{X}$, we have
$$
\ell(\wt\phi\circ\wt{\gamma}) \leq N \ell(\wt{\gamma}) .
$$
		Moreover, if $\wtgm$ passes through the origin, for every $\epsilon>0$ there exists $\wtgm':[a,b]\to \wt{X}$ that does not contain $\bbo$ such that $\ell(\wtgm)+\epsilon = \ell(\wtgm')$ and from definition of $\wtphi$ the curve $\iota_p\circ\wtgm'$ is well-defined. Thus without loss of generality we can assume $\bbo\notin \wtgm([a,b])$.
				
		Since $\phi$ is bi-Lipschitz, there exists a positive constant $L$ such that for any Lipschitz curve 
	$\gm:[a,b] \to (X,d)$ we have
	$$
	\frac{1}{L}\,\ell(\gm) \; \leq \; \ell(\phi\circ\gm) \; \leq \; L\,\ell(\gm).
	$$
	Therefore  
	$$
	d_\inn^{X'}(\phi(\bx),\bbo) \leq L\,C \,|\bx| \;\; {\rm and} \;\; d_\inn^X(\phi^{-1}(\bx'),\bbo) 
	\leq L\,C \,|\bx'|. 
	$$
	
	
	Let $\wtgm$ be a Lipschitz unit speed curve $[a,b] \to (\wt{X},d_\inn^{\wt{X}})$ and let 
	$\gm = \iop(\wtgm)$. Since $|\bx|^2 D_\bx \iop$ is an isometry at any $\bx\neq \bbo$, we deduce that $\gm$ is a Lipschitz mapping $[a,b] \to (X,d_\inn^X)$. Moreover, simple calculation shows that for a.e. 
	$t$ the following holds 
	$$
	|\dot\gm(t)| = \frac{1}{|\wtgm(t)|^2} = |\gm(t)|^2.
	$$
Thus  
		$$
	\left| \frac{\rd}{\rd t} (\wtphi\circ\wtgm) \right|
	= \left| \frac{\rd}{\rd t}(\ioq\circ\phi\circ\gm)\right| = 
	\frac{\left| \frac{\rd}{\rd t}(\phi\circ\gm ) \right| }{ |\phi\circ\gm|^2} \leq 
	\frac{L \cdot |\dot\gm(t)|}{|\phi\circ\gm(t)|^2}\leq L (LC)^2
	$$
almost everywhere and we proved $\wtphi$ is inner Lipschitz.
\end{proof}
Let $X$ be a closed definable subset of $\Rn$ containing the origin. 
Assuming that there exists $r_0$ small enough such that $X_{\leq r}$ is 
connected for any positive radius $r\leq r_0$. Using the LNE decomposition of Section 
\ref{section:L-RD}, 
one can easily observe that there exists a constant $C_r>0$ such that
$$
d_\inn^{X_{\leq r}}(\bx,\bbo) \;\leq\; C_r\, |\bx| \;\; \forall\;\bx\in X_{\leq r}.
$$
Moreover, $C_r \to 1$ as $r\to 0$. Using the inversion we conclude that if $X$ is closed, and definable, there exists a positive constant $C$ such that  
 $$
 d_\inn^X(\bx,\bbo) \;\leq\; C\, |\bx| \;\; \forall\;\bx\in X.
 $$

\begin{proof}[Proof of Theorem \ref{prop:LNE-model}] 
Let $S = \sgm_p(X)\cup N_p$ be the closure of the one-point compactification of $X$
in $\bS^p$ where $N_p$ is the north pole. Therefore there exists a definable compact connected subset $S'$ of $\Rq$  which is LNE and inner bi-Lipschitz homeomorphic to $S$
via $\phi:S\to S'$. Without loss of generality we can assume that $S'$ is contained in
$\bS^q$ and that $\phi(N_p) = N_q$ the north pole of $\bS^q$. 

The subset $X' := \sgm_q^{-1}(S'\setminus N_q)$ is definable and LNE by Theorem 
\ref{thm:main}. Let $\psi:= \sgm_q^{-1}\circ \phi \circ
\sgm_p|_X$, which is a homeomorphism from $X \to X'$. The differentiability of 
$\sgm_p$ and $\sgm_q^{-1}$ guarantees that the inner bi-Lipschitz behaviour of 
$\psi$ has to be checked only at infinity.   

Since the mapping $\iop\circ\sgm_p^{-1} : (\bS^p,N_p) \to (\Rp,\bbo)$ is a smooth 
diffeomorphism the mapping 
$$
\ioq^{-1} \circ\psi\circ \iota_p : (\iop(X),\bbo) \to (\ioq(X'),\bbo) 
$$
is inner bi-Lipschitz in a neighbourhood of the origin, therefore by Lemma
\ref{lem:inner-inversion} we deduce that $\psi$ is inner bi-Lipschitz in a neighbourhood of 
infinity.
\end{proof}

Let us state the generalized version of corollary of~\cite{BiMo}
\begin{corollary}
	Let $X$ be a closed definable set. Let $\mathcal{L}_X$ be the set
	of all definable sets bi-Lipschitz equivalent to $X$ with respect to the inner
	metric. We define a semiorder relation on $\mathcal{L}_X$ in the following way: $X \prec Y$  if
	there exists a map $F: X \to Y$ bi-Lipschitz with respect to the inner metric and
	Lipschitz with respect to the outer metric. Then $\mathcal{L}_X$ contains a unique (up to
	a bi-Lipschitz equivalence with respect to the outer metric) maximal element.
	This element is Lipschitz normally embedded.
\end{corollary}

The main result of \cite{BiMo} also implies that the Hausdorff distance $d_{\cH}$ between the set $X$ and an inner bi-Lipschitz homeomorphic LNE set $X'$ can be taken arbitrarily small (while  $\Rp$ is treated as a linear subspace of $\Rq$ and the dimension $q\geq p$ depends only on $X$). We cannot obtain such result in non-compact case by Corollary~\ref{cor:main}, see Example~\ref{exLNEnotclose} of the parabola. 
%
%
%
%
%
%
%
%
%
%
%
%
%
%
%
%
%
%
%
%
%
%
%
%
%
%
%
%
%
%
%
%
%
%
%
%
%
\section{Examples within and outside the definable 
world}\label{section:examples}
This last section complements the case of complex curves in
\cite[Section 9]{CoGrMi1} and points out occurrences of 
Theorem \ref{thm:main} outside the definable setting.

\medskip
Suppose we are compactifying the real plane $\R^2$ as the real projective
plane $\bP^2$, that is 
$$
\bP^2 = \R^2 \cup L_\infty,
$$
where $L_\infty$ is the line at infinity. If $[u:v:w]$ are coordinates 
over $\bP^2$, we identify $\R^2$ with the chart $w\neq 0$ so that 
$L_\infty = \{w=0\}$.

Suppose that $\bP^2$ is equipped 
with a given smooth Riemannian metric $\bg$. Let $d_\bg$ be the distance 
function induced by $\bg$. For $Z$ any subset of $\bP^2$, let $d_{\bg,Z}$
be the outer distance over $Z$ w.r.t. the distance $d_\bg$, 
and let $d_{\bg,\inn}^Z$ be the inner distance over $Z$ w.r.t. the Riemannian 
metric $\bg$. 
The subset $Z$ of $\bP^2$ is \em projectively Lipschitz Normally Embedded \em
(shorten to PLNE) if there exists a positive constant $C$
such that $d_{\inn,\bg}^Z \leq C\cdot d_{\bg,Z}$.

\begin{example}
The real affine parabola $\{(x,x^2), \; x\in \R\}$ is not locally
LNE at $\infty$ (as a subset of $\R^2$, by point (i) of Lemma 
\ref{lem:asympt-not-LNE}). Yet its closure $\{vw=u^2\}$ in $\bP^2$ is connected 
and non-singular, thus is PLNE (by \cite[Lemma  
2.8]{CoGrMi1}).
\end{example}


\begin{example}
The projective curve $X := \{v(w-u)(w+u) - w^3 = 0\}$ is irreducible, has a 
simple node at $[0:1:0]$, therefore it is PLNE (the real version  of 
\cite[Proposition 4.2]{CoGrMi1} is also true). Its affine part 
$X \setminus L_\infty$ consists of 
three connected components, the graphs 
$$
\Gm^* := \left\{\left(x,\frac{1}{(1-x)(1+x)}\right), \; x \in I^* \right\} 
$$
where $I^- = (-\infty,-1)$, $I^+ = (1,\infty)$ and $I^0 = (-1,1)$. It is
easy to check that $\Gm^\pm$ is LNE while $\Gm^0$ is not locally LNE at 
$\infty$ (by point (ii) of Lemma \ref{lem:asympt-not-LNE}).
\end{example}

\begin{example}\label{exLNEnotclose}
Let $X$ be the real parabola. Its LNE model is simply the real line $\R$. Any outer bi-Lipschitz embedding  
		of the line $\R$ in $\Rq$, $q\geq 1$, will have two distinct points at infinity. Since the parabola cannot be isometrically embedded into $\R$, thus consider embeddings of the parabola into $\Rq$, $q\geq 2$. 
		  Clearly, 
		as $R\to \infty$, for any outer bi-Lipschitz embedding $Y$ of the line $\R$ in $\Rq$ we have 
		$$
		d_\cH(X_{\leq R},Y_{\leq R}) = O(R).
		$$
\end{example}
Let us conclude with some non-definable examples.

\begin{example}
Let $e$ be a real number, and let $s_e:[1,\infty) \to \R$ be the function 
defined as $t\mapsto t^e \sin(t)$. Let $\gm_e :[1,\infty) \to \R^2$ be the 
smooth parametrization $t\mapsto (t,s_e(t))$ of $S_e$ the graph of the 
function $s_e$. Since the function $s_e$ oscillates at $\infty$, it is 
not definable in any o-minimal structure.

Using elementary estimates obtained from $\gm_e$ and $\gm_e'$, we verify
that $S_e$ is locally LNE at $\infty$ if and only if $e \leq 0$. We also 
check that the inversion $\clos(\iota_2(S_e))$ is locally LNE
at $\bbo$ if and only if $e\leq 0$, in other words if and only if $S_e$ is 
locally LNE at $\infty$. 
\end{example}

\begin{example}
Consider the following smooth plane curve, which is also not definable in
any o-minimal structure:
$$
\cS := \clos(\{e^t\cdot e^{i2\pi t}, \; t\in \R\}).
$$
A simple computation allows to check that $\cS$ is LNE with LNE constant 
$\sqrt{1+4\pi^2}$. A similar computation also shows that 
$\clos(\iota_2(\cS))$ is LNE with LNE constant $\sqrt{1+4\pi^2}$. Thus $\cS$ is LNE both in $\R^2$ and under the inverse of stereographic projection in $\bS^2$.
%
%
%
%

As a striking contrast, working with $\bP^2$ as the 
compactification of $\R^2$, 
we  
observe that $\cS \cup L_\infty$ is the closure of $\cS$ taken in $\bP^2$. 
Since $\cS$ accumulates on the whole $L_\infty$, its length w.r.t. $\bg$ is 
infinite, therefore we deduce 
$$
\sup\left\{\frac{ d_{\inn,\bg}^\cS (\bx,\bx') }{ d_{\bg,\cS}(\bx,\bx')}, \;
\bx,\bx' \in \cS, \; \bx \neq \bx'
\right\} = \infty,
$$
in other words $\cS$ is not PLNE.
\end{example}
%
%
%
%
%
%
%
%
%
%
%
%
%
%
%
%
%
%
%
%
%
%
%
%
%
%
%
%

%
%
%
%
%
%
%
%
%
%
%
%
%
%
%
%
%
%
%
%
%

%
%
%
%
%
%
%
%
%
%
%
%
%
%
%
%
%
%
%
%
%
%
%
%
%
%
%
%
%
%

%
%
%
%
%
%
%
%
%
%
%
%
%
%
%
%
%
%
%
%
%
%
%
%
%
%
%

\end{document}